\documentclass[conference]{IEEEtran}
\IEEEoverridecommandlockouts
\usepackage{cite}
\usepackage{graphicx}
\usepackage{url}
\usepackage{lipsum}
\usepackage{color}
\usepackage{amsmath}
\usepackage[english]{babel}
\usepackage{amsthm}
\newtheorem{theorem}{Theorem}[section]

\newtheorem{proposition}[theorem]{Proposition}

\usepackage{mathtools}
\usepackage{array}
\usepackage{multirow}
\usepackage{makecell}
\usepackage{float}
\usepackage{setspace}
\usepackage{supertabular,booktabs}
\usepackage{amssymb}
\usepackage{algorithm}
\usepackage{algpseudocode}
\usepackage{comment}
\newboolean{showcomments}
\setboolean{showcomments}{true}

\usepackage{amsfonts}
\usepackage{multicol}
\usepackage[top=1.1cm, bottom=1.1cm, left=1.3cm, right=1.3cm]{geometry}
\allowdisplaybreaks[4]
\usepackage{nomencl}
\usepackage{caption}
\usepackage{subcaption}
\newcommand{\changed}[1]{\textcolor{black}{#1}}

\def\BibTeX{{\rm B\kern-.05em{\sc i\kern-.025em b}\kern-.08em
		T\kern-.1667em\lower.7ex\hbox{E}\kern-.125emX}}
\begin{document}
	
	\title{\huge Feasible Region for DERs in Unbalanced Distribution Networks with Uncertain Line Impedances\\
		\thanks{The authors would like to acknowledge support from the CSIRO Strategic Project on Network Optimisation \& Decarbonisation under Project Number: OD-107890.}
	}
	
	\author{\IEEEauthorblockN{Bin Liu {\it Member, IEEE}}
		\IEEEauthorblockA{Energy Systems Program, Energy Centre\\
			Commonwealth Scientific and Industrial Research Organisation\\
			Mayfield West, NSW 2304, Australia\\
			brian.liu@csiro.au}
		\and
		\IEEEauthorblockN{Jin Ma {\it Member, IEEE}}
		\IEEEauthorblockA{School of Electrical and Information Engineering\\
			The University of Sydney \\
			Camperdown, NSW 2006, Australia\\
			j.ma@sydney.edu.au}
	}
	
	\maketitle
	
	\begin{abstract}
		The rapid development of distributed energy resources (DERs) has brought many challenges to the operation of distribution networks in recent years, where operating envelopes (OEs), as a key enabler to facilitate DER integration via virtual power plants, have attracted increasing attention. Geometrically, OEs are inherently linked to the feasible region (FR), which depicts an area that quantifies the permissible operational ranges for all DERs. This paper studies how to calculate FR for DERs based on a linear unbalanced three-phase power flow model while considering uncertainties from measured/estimated line impedances, leading to the robust feasible region (RFR). With an initialised FR, the proposed solution algorithm aims to seek the RFR by solving a non-convex bilinear problem, formulated based on Motzkin Transposition Theorem, and updating the RFR alternately until reaching convergence. Two cases, including a 2-bus illustrative network and a representative Australian network, are tested to demonstrate the effectiveness of the proposed approach.
	\end{abstract}
	
	\begin{IEEEkeywords}
		DERs, distribution network, operational envelopes, robust optimisation, unbalanced OPF, uncertainty modelling
	\end{IEEEkeywords}
	
	\section{Introduction}
	Distributed energy resources (DERs), has been rapidly growing worldwide, and particularly in Australia in recent years \cite{AustralianEnergyCouncil2020}. High penetration of DERs is a crucial enabler to decarbonise the energy sector through more engagement of residential customers. However, due to relatively low visibility and controllability, high DER penetration could lead to, for example, over/under-voltage and minimum operational demand issues \cite{Review-01, Review-02,aemo-min-demand}. Addressing such issues needs close coordination among transmission network operators (TSOs), distribution network operators (DSOs), emerging DER aggregators in operating virtual power plants (VPPs) and market operators, and potential future coordination architectures were investigated in \cite{EnergyNetworksAustralia2020}.
	
	Operating envelope (OE), which specifies the operational range for each DER device that is permissible within the network operational limits at a given point in time and location \cite{DEIP2022, Petrou2021}, is identified as a key enabler in potential future power system architectures and has gained increasing interest from both industry and academia \cite{DEIP2022,Petrou2021,project_edge,project_symphony,ESB_DER}. 
	\changed{
		In general, OEs may not be necessary for a network with low penetration of DERs since over-voltage issues can be neglected and under-voltage issues when peak load scenarios occur can be well addressed. However, when the penetration of DERs continues increasing in the future, OEs can be used by DSOs to manage the export and import limits for DERs based on available real-time network hosting capacity and facilitate DER participation in the electricity market to provide network services \cite{Liu2021}}.
	Substantial advances have been made in developing approaches to calculating OEs in recent years, including the unbalanced three-phase power flow (UTPF)-based approach, where OEs are calculated iteratively by solving a series of PF problems with varying export/import limits \cite{Blackhall2020,project_edge,project_symphony,BL_ieee_access}, the unbalanced three-phase optimal power flow (UTOPF)-based approach, where OEs of all DERs are calculated simultaneously subject network constraints \cite{To2013,Petrou2021,Liu2022_doe}, and machine learning (ML)-based approach, which is useful for networks with low visibility \cite{Ochoa2022,DEIP2022}.
	However, uncertainties naturally exist in power grids and may come from load forecast errors, inaccurate estimations or measurements for line impedances and topology. Such uncertainties can potentially undermine the accuracy of FR and, as a result, lead to unreliable and non-robust OEs. To address this issue, this paper proposes an approach to calculate the \emph{robust} FR (RFR) for DERs in an unbalanced distribution network while considering uncertainties from line impedances based on a linear UTPF model. Such uncertainties are formulated as box sets on either self and mutual impedances directly or on the positive/negative and zero sequence impedances of all types of line codes. With formulated uncertainties and an initialised FR, seeking the RFR is achieved by alternately solving a \emph{max-min} problem formulated based on Motzkin Transposition Theorem (MTT). Both a 2-bus illustrative network and a representative Australian network are tested to demonstrate the effectiveness of the proposed approach.
	
	It is noteworthy that the proposed approach only applies when a linear UTPF model is used, and we admit that errors inherently exist for such formulations and are unavoidable. Hence, improving the accuracy of linear UTPF will help improve the effectiveness of the RFR. \changed{Interested readers are referred to \cite{liu2022robust} for a comparative study of FRs calculated based on a linear UTPF model and an exact UTOPF model.
		Moreover, benefits of calculating FR include: (1) assisting in OE calculation noting an OE allocation strategy usually corresponds to an \emph{boundary point} of FR, either robust or not, and an OE strategy can be conveniently calculated with updated FR parameters and a pre-defined objective function; and (2) helping assess the security risk of operating a network with high penetration of DERs without repeatedly running power flows, as discussed in \cite{Wei2015a} for transmission networks.}
	
	The remainder of the paper is organised as follows. Section \ref{sec_02} presents the formulation of FR both with and without considering uncertainties from line impedances. Section \ref{sec_03} discusses the solution algorithm to seeking the RFR, followed by case studies in Section \ref{sec_04}, and the paper is concluded in Section \ref{sec_05}.
	
	\section{Feasible Region for DERs}\label{sec_02}
	\subsection{UTPF with operational constraints}
	UTPF is an essential tool for the distribution network analysis to capture its operational states \cite{RN37,RN19}. Although a typical low-voltage distribution network may contain four wires, including the neutral conductor, this paper assumes that the neutral conductor is well grounded and the line parameters for the equivalent three-phase networks can be calculated via Carson's Equations \cite{RN116}.
	
	Based on its current and voltage formulation, the typical UTPF model with operational constraints can be formulated as \eqref{utpf-cons}.
	\begin{subequations}\small
		\label{utpf-cons}
		\begin{eqnarray}
			\label{utpf-cons-01}
			V^{\phi}_{i_\text{ref}}=V^{\phi}_{0}~~\forall \phi\\
			\label{utpf-cons-02}
			V^{\phi}_{i}-V^{\phi}_{j}=\sum\nolimits_\psi{z_{ij}^{\phi\psi}I^{\phi}_{ij}}~~\forall \phi,\forall ij\\
			\label{utpf-cons-03}
			\sum_{n:n\rightarrow i}{I^{\phi}_{ni}}-\sum_{m:i\rightarrow m}{I^{\phi}_{im}}=\sum_m{\frac{\mu_{\phi,i,m}(P^\text{}_{m}-\text{j}Q^\text{}_m)}{(V^\phi_{i})^*}}~~\forall \phi,\forall i\\
			\label{utpf-cons-05}
			V^\text{min}_{i}\le |V^{\phi}_{i}|\le V^\text{max}_{i}~ \forall \phi,\forall i
		\end{eqnarray}
	\end{subequations}
	where
	$i_\text{ref}$ is the index of the reference bus and $V^{\phi}_{0}$ is its fixed voltage at phase $\phi$ (known parameter).
	$V^{\phi}_{i}$ is the voltage of phase $\phi$ at node $i$.
	$I^{\phi}_{ni}$ is the current in phase $\phi$ of line $ni$: flowing from bus $n$ to bus $i$.
	$P_m$ is the active power demand of customer $m$ while $Q_m$ is the reactive power demand. For simplicity, all $P_m$ and $Q_m$ are treated as variables in the formulation. However, they will be fixed as the forecasted values if they are uncontrollable.
	$\mu_{\phi,i,m}\in\{0,1\}$ is a parameter indicating the phase connection of customer $m$ with its value being $1$ if it is connected to phase $\phi$ of bus $i$ and being $0$ otherwise.
	$V^\text{min}_{i}$ and $V^\text{max}_{i}$ are the lower and upper limits of $|V_{i}|$, respectively.
	
	In the formulation, \eqref{utpf-cons-01} specifies the voltage at the reference bus. \eqref{utpf-cons-02} is the voltage drop equation in each line, and Kirchhoff's current law is formulated as \eqref{utpf-cons-03}. Only constraints on voltage magnitudes are considered in this paper, and they are expressed as \eqref{utpf-cons-05}.
	
	As discussed previously, a linear UTPF (LUTPF) model will be used, where the physical characteristics of a real distribution network, including the reality in most cases that differences of voltage angles in each phase are sufficiently small \cite{RN38,RN50} and nodal voltages throughout the network are close to $1.0~p.u.$, will be taken. \changed{Note that the linearisation techniques in an unbalanced distribution network is substantially different from the direct-current power flow (DC-PF) for an balanced transmission network, where for each line $ij$ all voltage magnitudes are replaced by 1.0 p.u., and $\cos\theta_{ij}$ and $\sin\theta_{ij}$ are approximated by 1.0 and $\theta_i-\theta_j$, respectively.} Specifically, \eqref{utpf-cons-03} can be linearised by fixing $V^\phi_i$ in the denominator on the right-hand side, while \eqref{utpf-cons-05} will be approximated by a set of linear constraints. Details of the linearisation approach can be found in \cite{BL-isgt-asia} and are omitted here for simplicity. For the convenience of later discussions, LUTPF can be expressed in the following compact form,
	\begin{subequations}\small\label{utpf-compact}
		\begin{eqnarray}
			\label{utpf-compact-01}
			Ap+Bq+Cl=b\\
			\label{utpf-compact-02}
			Dv+El=d\\
			\label{utpf-compact-03}
			Fv\le f
		\end{eqnarray}
	\end{subequations}
	where
	$p$ and $q$, which are independent variables, represent the active and reactive powers;
	$l$ represents the vector consisting of state variables of line currents;
	$v$ is a vector consisting of state variables related to nodal voltages;
	$A,B,C,b,D,E,d,F$ and $f$ are constant parameters with appropriate dimensions;
	
	In the formulation,
	\eqref{utpf-compact-01} links back to \eqref{utpf-cons-03} after linearisation and represents the relations between line currents and residential demands;
	\eqref{utpf-compact-02} represents the linearised power flow equations that link the power demands and currents running in all lines, i.e. \eqref{utpf-cons-01}-\eqref{utpf-cons-02},
	and \eqref{utpf-compact-03} represents all the operational constraints after the linearisation.
	
	\subsection{FR without network uncertainties}\label{fr_fixed}
	Noting that only $p$ and $q$ are independent variables, \eqref{utpf-compact} defines the FR for $(p,q)$. This paper focuses on deriving the feasibility region with $p$ as the independent variable, as most residential homes consume active power loads. Correspondingly, the FR depicts an area in which, if all realised values of $p$ fall within, the operational security of the network can be guaranteed \cite{Wei2015a,Riaz2022}. After removing state variables $v$ and $l$, the FR for $p$ can be expressed as \eqref{fr}.
	{\small \begin{eqnarray}\label{fr}
			\mathcal{F}(q)=\{p|Hp\le h\}
			=\left\{p\left\vert\begin{matrix}
				FD^{-1}EC^{-1}Ap\le f \\-FD^{-1}(d-EC^{-1}b+EC^{-1}Bq) \\
			\end{matrix}\right.\right\}
	\end{eqnarray}}
	where
	$H=FD^{-1}EC^{-1}A$ and $h=f-FD^{-1}(d-EC^{-1}b\\+EC^{-1}Bq)$.
	
	Obviously, the $\mathcal{F}(q)$ is a function of $q$, and also a polyhedron if $A,B,C,b,D,E,d,F$ and $f$ are fixed\footnote{For the convenience of later discussions, parameters will not be specifically indicated that it is a function of $q$ in following paragraphs.}. However, uncertainties inherently exist in real power systems and can arise from errors in forecasting load profiles for passive customers, incorrectly recorded phase connectivity for residential customers and network topology, and inaccuracy in measuring or estimating line impedances. This paper will focus on uncertainty modelling for line impedances and their impacts on the FR.
	
	\subsection{FR considering uncertainties from line impedances}
	\subsubsection{Uncertainty modelling}
	In this paper, uncertainties on line impedances will be modelled as box sets either on both self-impedances and mutual impedances directly, as presented in \eqref{uset-01}, or on both the positive/negative and zero-sequence impedances, as indicated by \eqref{uset-02}.
	{\small\begin{eqnarray}\label{uset-01}
			z_{ij}^{\phi\psi,\text{min}}\le z_{ij}^{\phi\psi}\le z_{ij}^{\phi\psi,\text{max}}~\forall ij,\forall \phi,\forall\psi
	\end{eqnarray}}
	where
	$z_{ij}^{\phi\psi,\text{min}}$ and $z_{ij}^{\phi\psi,\text{max}}$ are the low and upper bounds for the $z_{ij}^{\phi\psi}$, respectively.
	\begin{subequations}\label{uset-02}
		\begin{eqnarray}
			z_{ij}^{+,\text{min}}\le z_{ij}^{+}=z_{ij}^{-}\le z_{ij}^{+,\text{max}}~\forall ij\\
			z_{ij}^{0,\text{min}}\le z_{ij}^{0}\le z_{ij}^{0,\text{max}}~\forall ij\\
			z_{ij}^{\phi\psi}=(2z_{ij}^{+}+z_{ij}^{0})/3~\forall ij,\forall \phi,\forall \psi~(\psi=\phi)\\
			z_{ij}^{\phi\psi}=(-z_{ij}^{+}+z_{ij}^{0})/3~\forall ij,\forall \phi,\forall \psi~(\psi\neq\phi)
		\end{eqnarray}
	\end{subequations}
	where
	$z_{ij}^{+},z_{ij}^{-}$ and $z_{ij}^{0}$ are the positive, negative and zero-sequence impedances for line $ij$, respectively, with their lower or upper limits indicated by $z_{ij}^{+,\text{min}},z_{ij}^{+,\text{max}},z_{ij}^{0,\text{min}}$ and $z_{ij}^{0,\text{max}}$.
	
	For the convenience of later discussion, $z$ is used to represent line impedances, and its uncertainty is compactly formulated as $z^\text{min}\le z\le z^\text{max}$. Noting that $z_{ij}^{\phi\psi}$ are coefficients only for the \emph{current} terms in \eqref{utpf-cons-02}-\eqref{utpf-cons-03}, in the compact form \eqref{utpf-compact}, such uncertainties will be broadcasted to $E$ and also $H$ and $h$. Specifically, taking $E=f(z)$ to represent the linear mapping from uncertain line impedances $z$ to the matrix $E$, the uncertainty set for the pair $(H,h)$ can be formulated as
	{\small\begin{eqnarray}\label{uset-Hh}
			\mathcal{H}=\left\{(H,h)\left\vert\begin{matrix}
				H=FD^{-1}f(z)C^{-1}A                    \\
				h=f-FD^{-1}(d-f(z)C^{-1}b+f(z)C^{-1}Bq) \\
				z^\text{min}\le z\le z^\text{max}       \\
			\end{matrix}\right.\right\}
	\end{eqnarray}}
	
	\subsubsection{Mathematical formulation}
	With uncertain $E, H$ and $h$ in \eqref{fr}, to secure the network operation, it is critical to find a more \emph{conservative} yet more reliable and robust FR for all DERs' generation, which will be a subset of $\mathcal{F}$. The more reliable and robust FR, which is named \emph{robust feasible region (RFR)} in this paper, can be mathematically defined as
	{\small\begin{eqnarray}\label{rfr-01}
			\mathcal{R}=\{p|Gp\le g\}=\cap_{(H,h)\in\mathcal{H}}\mathcal{F}=\{p|Hp\le h~\forall (H,h)\in\mathcal{H}\}
	\end{eqnarray}}
	which is equivalent to
	{\small\begin{eqnarray}
			\label{rfr-02}
			\{p|Gp\le g\}\cap\{p|H_{i,:}p>h_i\}=\emptyset~\forall (H,h)\in\mathcal{H},\forall i
	\end{eqnarray}}
	where $H_{i,:}$ represents the $i-th$ row of $H$ and $h_i$ is the $i-th$ element in $h$.
	
	Obviously, the $\mathcal{R}$ is also a polyhedron. However, noting that uncertainties from $H$ and $h$ are depicted by a closed set, there are indefinite constraints in \eqref{rfr-02} and a suitable solution algorithm needs to be developed to get the exact formulation of the RFR.

	\section{Solution algorithm}\label{sec_03}
	\subsection{Reformulation of \eqref{rfr-02} to an optimisation problem}
	The reformulation is based on the MTT \cite{Ben_Israel2001}, where we have the following proposition for our case.
	\begin{proposition}
		The formulation \eqref{rfr-02} is equivalent to
		{\small\begin{eqnarray}\label{rfr-03}
				G^Tx_{:,i}=H^T_{i,:},~~x_{:,i}\ge 0,~~g^Tx_{:,i}\le h_i~\forall i
		\end{eqnarray}}
		where $x$ is a matrix variable and and $x_{:,i}$ is its $i-th$ column.
	\end{proposition}
	\begin{proof}
		Given the matrix $A$, vectors $b,d$ and a number $c$, MTT indicates that only one of the following two statements holds.
		\begin{enumerate}
			\item The system $Ax\le b, d^Tx\le c$ has a solution $x$;
			\item The system
			{\small\begin{eqnarray}\label{mtt_c}
					y\ge 0,~~z\ge 0,~~A^Ty+z\cdot d=0,~~b^Ty+c\cdot z< 0
			\end{eqnarray}}
		\end{enumerate}
		has a solution $(y,z)$.
		
		With \textit{MTT},  $\mathcal{R}=\{p|Gp\le g\}$ being a subset of $\mathcal{F}=\{p|Hp\le h\}$ is equivalent to $\mathcal{R}\cap\{p|H_{i,:}p>h_i\}=\emptyset~(\forall i)$. Further, $\mathcal{R}$ and $H_{i,:}p>h_i$ can be reformulated as \eqref{mtt_01_01} and \eqref{mtt_01_02}, respectively.
		\begin{subequations}\small
			\begin{eqnarray}\label{mtt_01_01}
				\left[G,~-G,~I_s\right]\left[\begin{matrix}p^+\\p^-\\s\end{matrix}\right] + 1\cdot (-g)=0, ~~
				\left[\begin{matrix}p^+\\p^-\\s\end{matrix}\right]\ge 0\\
				\label{mtt_01_02}
				\left[-H_{i,:},~H_{i,:},~0_s^T\right]\left[\begin{matrix}p^+\\p^-\\s\end{matrix}\right]+h<0
			\end{eqnarray}
		\end{subequations}
		where $y=y^+-y^-$ and $s$ is an auxiliary vector variable.
		
		After applying MTT, $\mathcal{R}\cap\{p|H_{i,:}p>h_i\}=\emptyset$ is equivalent to
		{\small\begin{eqnarray}
				G^Ty=H_{i,:}^T,~y\ge 0,~g^Ty\le h_i
		\end{eqnarray}}
		with $y$ being a vector variable and, accordingly, \eqref{rfr-02} can be equivalently reformulated as
		{\small\begin{eqnarray}
				G^Tx_{:,i}=H^T_{i,:},~~x_{:,i}\ge 0,~~g^Tx_{:,i}\le h_i~\forall i
		\end{eqnarray}}
		where $x$ is now a matrix variable.
	\end{proof}
	
	Based on the proposition, if \eqref{rfr-02} holds for fixed $H$ and $h$, the optimal objective value of \eqref{rfr-04} must be 0.
	\begin{subequations}\label{rfr-04}
		\small
		\begin{eqnarray}
			\label{rfr-04-01}
			\min_{x,\Delta h}{1^T\Delta h}\\
			s.t.~~~G^Tx=H^T,~g^Tx\le h^T+\Delta h^T\\
			x\ge 0,~\Delta h\ge 0
		\end{eqnarray}
	\end{subequations}
	
	Similarly, if \eqref{rfr-02} holds for all possible $H$ and $h$, the optimal objective value of \eqref{rfr-05} will be 0.
	\begin{subequations}\label{rfr-05}
		\small
		\begin{eqnarray}
			\label{rfr-05-01}
			\max_{(H,h)\in \mathcal{H}}\min_{x,\Delta h}{1^T\Delta h}\\
			s.t.~~~G^Tx_{:,i}=H^T_i~\forall i~~~(\alpha_i)\\
			x_{:,i}\ge 0,~\Delta h_i\ge 0~\forall i\\
			g^Tx_{:,i}\le h_i+\Delta h_i~\forall i~~~(\beta_i)
		\end{eqnarray}
	\end{subequations}
	where
	$\alpha_i$, as a vector, and $\beta$, as a scaler, are Lagrange multipliers.
	
	Based on strong duality theory, \eqref{rfr-05} is equivalent to
	\begin{subequations}\label{rfr-06}
		\small
		\begin{eqnarray}
			\text{BiLP-Max:~~} \max_{(H,h)\in \mathcal{H},\alpha,\beta}{\sum\nolimits_i{(-\alpha_i^TH^T_{i,:}-\beta_i h_i})}\\
			s.t.~~~0\le \beta_i \le 1~\forall i\\
			\alpha_i^TG^T+\beta_i g^T\ge 0~\forall i
		\end{eqnarray}
	\end{subequations}
	
	Obviously, \eqref{rfr-06} is a non-convex bilinear programming problem, where several solution approaches are available. One of them is reformulating \eqref{rfr-06} into a mixed-integer linear programming (MILP) problem noting that optimal solutions of $(H,h)$ always fall on one extreme point of $\mathcal{H}$. However, such an approach could become computationally intractable if there is a large number of uncertain variables \cite{Liu2019,Wei2015a}. Another approach is based on iteration-based linear programming (ITLP), where a series of linear programming (LP) problems will be solved by alternately fixing $(H,h)$ and $(\alpha,\beta)$ in \eqref{rfr-06} until reaching convergence. Although the ITLP-based approach cannot guarantee a global solution, it has been demonstrated with good performance with a sufficient number of and properly selected starting points for $(H,h)$ \cite{Wei2015a,Liu2019}. Noting that the number of uncertain parameters can be large in our case, the ITLP-based approach will be employed to solve \eqref{rfr-05}.
	
	\subsection{Updating $\mathcal{R}$ and removing redundant constraints}
	After solving \eqref{rfr-06}, a set of linear inequalities will be generated if the objective value is positive and added to update $\mathcal{R}$. With reported solutions, say $(H^*,h^*)$, from solving \eqref{rfr-06}, the RFR will be updated as
	{\small\begin{eqnarray}\label{sol-01}
			\text{U-RFR:~~}\mathcal{R}=\mathcal{R}\cap \{p|H^*p\le f^*\}
	\end{eqnarray}}
	
	Further, to reduce the computational burden in solving \eqref{rfr-06}, before conducting the next iteration, redundant constraints in the updated $\mathcal{R}$ will be removed first \cite{Fukuda2014}. In summary, a schematic diagram showing the process of seeking the RFR for DERs is given in Fig.\ref{fig_digram_robustFR}.
	\begin{figure}[htpb!]
		\centering\includegraphics[scale=0.7]{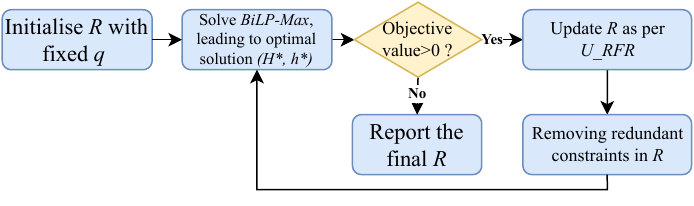}
		\caption{The schematic diagram of calculating the robust feasible region for DERs.}
		\label{fig_digram_robustFR}
	\end{figure}
	
	Several remarks on calculating the RFR for DERs are given below.
	\begin{enumerate}
		\item Noting that in a real network, the initial value of $\mathcal{R}$, denoted as $\bar{\mathcal{R}}$, can be initialised after fixing $z$ as its measured or estimated value.
		\item Although controllable reactive power $q$ can be regarded as a variable, it is fixed as a constant when calculating the RFR in this paper.
		\item Properly selecting starting points for $(H,h)$ is critical to ensure the quality of the solutions. In this paper, all projected points from $(\widetilde H,\widetilde h)$ to its uncertainty set boundaries will be used as starting points, which, as demonstrated in \cite{Wei2015a}, can lead to a high-quality solution for \eqref{rfr-06}.
		\item After solving \eqref{rfr-05}, several solutions, which link back to the various starting point for $(H,h)$, can be reported. In such cases, $\mathcal{R}$ will be updated by considering all solutions in this paper.
		\item \changed{More uncertainties being considered will lead to various uncertain parameters in \eqref{utpf-compact} and hence various complexity and tractability in solving \eqref{rfr-05} or \eqref{rfr-06}. For example, errors in forecasting load profiles will lead to uncertainties in $b$, and uncertain network topology will result in uncertainties in both $C$ and $E$. Properly quantifying the uncertainty of such parameters and seeking efficient approaches to solving the derived optimisation problems, although beyond the scope of this paper, fall within our future research interest.}
		\item All LP problems in this paper are formulated in \texttt{Julia} \cite{bezanson2017julia} with data parsed by \texttt{PowerModelsDistribution.jl} \cite{pmd_ref}, and solved by Gurobi \cite{gurobi} on a laptop with Intel Core i7-1185G7 @ 3.00GHz and 16 GB RAM. It is noteworthy that parallel computing techniques can potentially accelerate computational efficiency, noting that \eqref{rfr-06} with different starting points can be solved simultaneously. However, investigating how parallel computing techniques could improve computational performance is beyond the scope of this paper and falls within our future research interest.
	\end{enumerate}

	\section{Case Study}\label{sec_04}
	\subsection{Case setup}
	In this section, two distribution networks, one of which is a 2-bus illustrative network and the other one is a representative Australian network, will be studied.
	For the illustrative distribution network, an ideal balanced voltage source with the voltage magnitude being $1.0~p.u.$ is connected to bus 1.
	A three-phase line connects bus 1 and bus 2, and its recorded impedance matrix is
	\begin{equation*}
		\bar{z}_{12}=\left[\begin{smallmatrix}
			0.3465+\text{j}1.0179 & 0.1560+\text{j}0.5017 & 0.1580+\text{j}0.4236\\
			0.1560+\text{j}0.5017 & 0.3375+\text{j}1.0478 & 0.1535+\text{j}0.3849\\
			0.1580+\text{j}0.4236 & 0.1535+\text{j}0.3849 & 0.3414+\text{j}1.0348
		\end{smallmatrix}\right] \Omega
	\end{equation*}
	and, given a number $\delta\in[0,1]$ representing the uncertainty level, we also assume the true line impedances are within the range $[1-\delta, 1+\delta]$ of its recorded value, i.e.
	{\small\begin{eqnarray}
			(1-\delta)\bar{z}_{12}\le z_{12}\le (1+\delta)\bar{z}_{12}
	\end{eqnarray}}
	and $\delta$ is set as 20\% for the 2-bus network.
	
	Three single-phase customers, numbered 1 to 3, are connected to phase $b$, $a$ and $c$ of bus 2, respectively, where all reactive powers are fixed at 0 kVar, and the active power of customer 2 is fixed at -2.0 kW. Moreover, both customers' default export/import limits are set as 20 kW. For this 2-bus network, only voltage magnitude constraints are considered for bus 2, where the lower and upper bounds for each phase are set as $0.95~p.u.$ and $1.05~p.u.$, respectively.
	
	The topology and network parameters of the representative Australian distribution network can be found in \cite{liu2022robust} and \cite{LVFT_data}, where the transformer has been changed to $Y_n/Y_n$ connection with $R$\%=5 and $X$\%=7. The voltage at the reference bus is set as $[1.0,~1.0e^{-\text{j}\frac{2\pi}{3}},~1.0e^{\text{j}\frac{2\pi}{3}}]^T$ while the upper and lower limits on voltage magnitudes of all the other buses are set as 1.05 $p.u.$ and 0.95 $p.u.$, respectively.
	Similar to the 2-bus network, the reactive powers of all customers are set to 0 kVar, and the active powers of 10 customers (customers: ``1", ``10", ``11", ``12", ``13", ``14", ``15", ``16", ``17", ``18"), where default export and import limits are set as 5 kW and 6 kW, respectively, form the independent variables in calculating the FR. In contrast, demands of all the other active customers are fixed at the values provided in \cite{LVFT_data}.
	In the Australian network, line impedances are generated from 5 line codes, where uncertainties are assumed on both their positive/negative and zero-sequence impedances ($\delta$=10\%).

	\subsection{The 2-bus illustrative distribution network}
	As only one line exists in the 2-bus illustrative network, there are 12 independent parameters with uncertainties, leading to a total number of 24 initial points by projecting $\bar z$ to the boundaries of its uncertainty set for solving \eqref{rfr-06}. 
	The proposed approach takes 42.26 seconds in total after three iterations to report the RFR, where the number of reported solutions and the computational time for each iteration are presented in Fig.\ref{fig_sol_2bus}. As shown in the figure, the number of unique solutions reported from solving the BiLP-Max problem becomes smaller while the objective value decreases along with increasing iterations. Regarding the computational time, the solver takes most of the time in solving BiLP-Max while removing redundant constraints is much more efficient, averagely taking 0.2 seconds in each iteration.
	\begin{figure}[htbp!]
		\centering
		\begin{subfigure}[b]{0.23\textwidth}
			\centering\includegraphics[width=\textwidth]{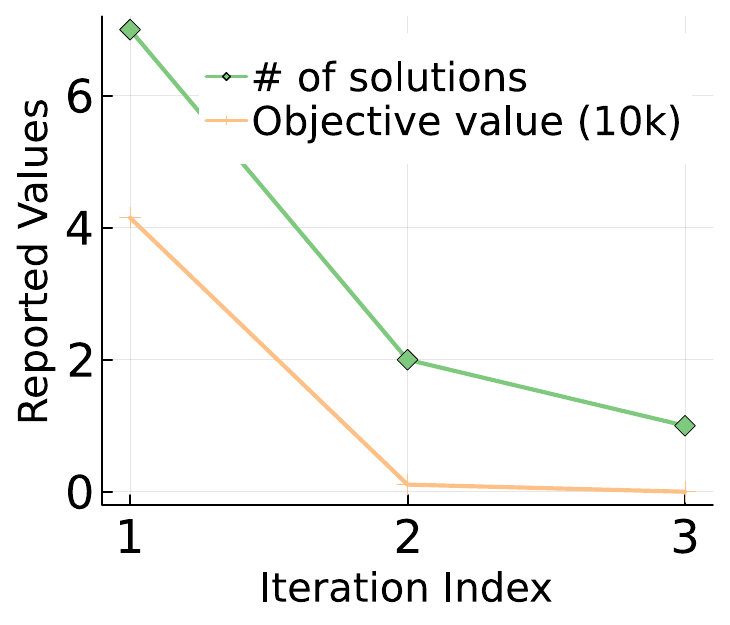}
			\label{fig_sol_2bus_fea}
		\end{subfigure}
		\hfill
		\begin{subfigure}[b]{0.23\textwidth}
			\centering\includegraphics[width=\textwidth]{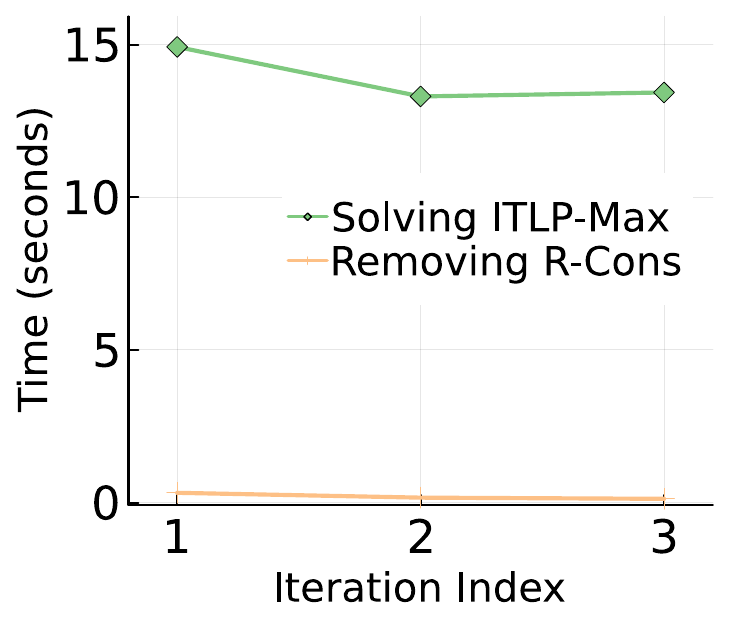}
			\label{fig_comput_time_2_bus}
		\end{subfigure}
		\hfill
		\caption{Computational and solution information for the 2-bus illustrative network.}
		\label{fig_sol_2bus}
	\end{figure}
	
	The initial FR, the finally reported RFR, and FRs calculated with randomly generated scenarios for $z$ are presented in Fig.\ref{fig_FRs_2_bus}, where the calculated RFR is a subset of both the randomly generated FRs and the initialised FR, demonstrating its robustness.
	\begin{figure}[htpb!]
		\centering\includegraphics[scale=0.2]{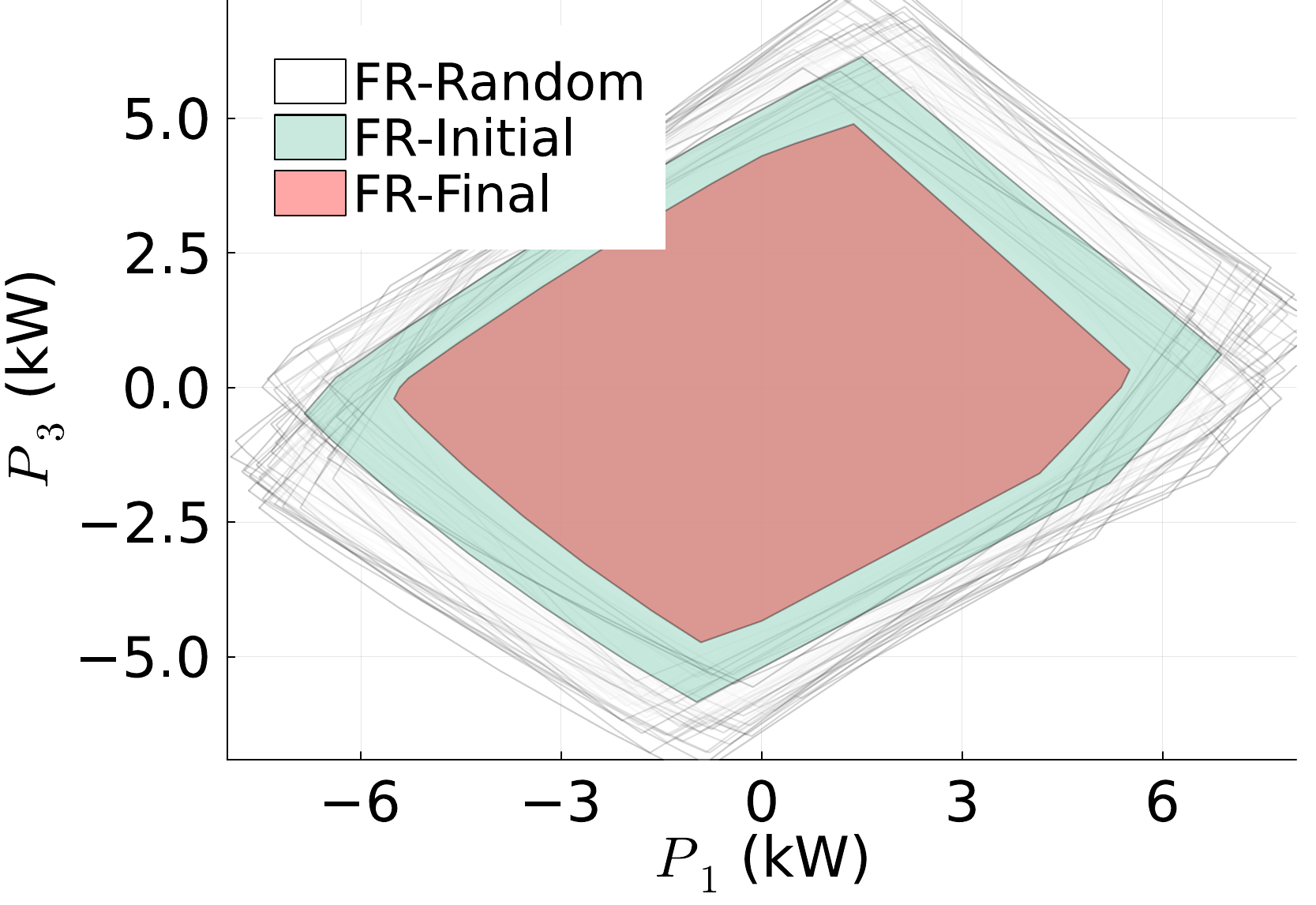}
		\caption{FRs calculated for the 2-bus illustrative network (R-Cons: Redundant Constraints).}
		\label{fig_FRs_2_bus}
	\end{figure}
	
	Moreover, the percentage of random FRs containing the RFR and initialised FR are also calculated, where the number is 100\% for RFR and 0\% for the initialised FR. As a higher percentage implies higher reliability and robustness of FR against uncertainties, such results further demonstrate the effectiveness of the proposed approach.
	
	\subsection{The real Australian network}
	The Australian network is used to test the computational efficiency of the proposed approach in a real network. Simulation results, including the number of reported solutions, objective value and computational time for each iteration, are presented in Fig.\ref{fig_sol_aus_J}.
	\begin{figure}[htbp!]
		\centering
		\hfill
		\begin{subfigure}[b]{0.23\textwidth}
			\centering\includegraphics[width=\textwidth]{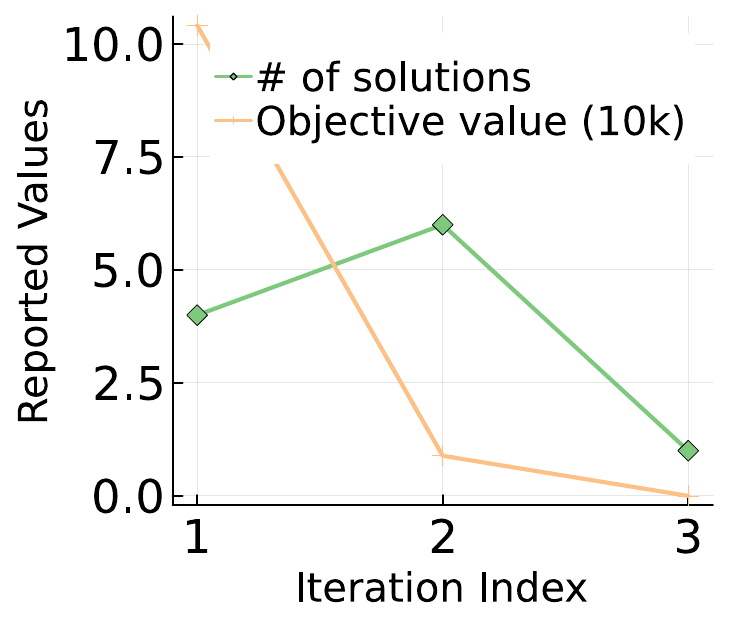}
			\label{fig_sol_aus_J_fea}
		\end{subfigure}
		\hfill
		\begin{subfigure}[b]{0.23\textwidth}
			\centering\includegraphics[width=\textwidth]{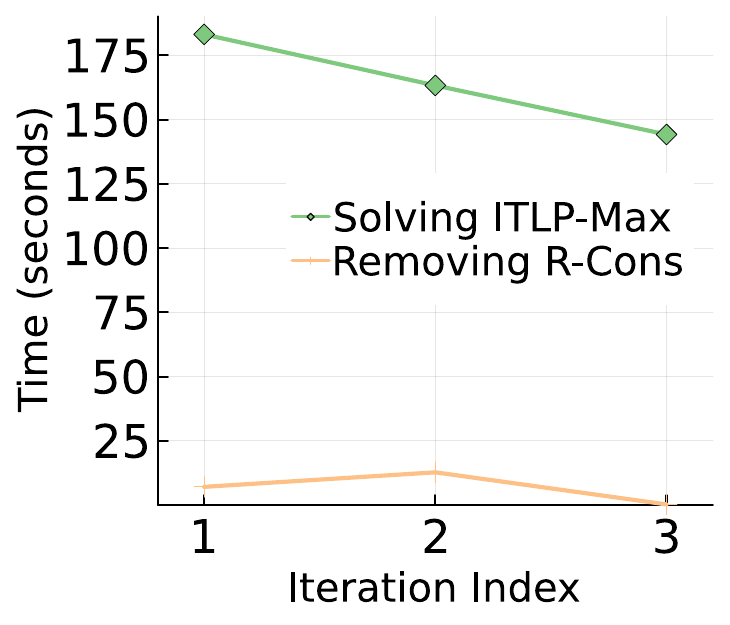}
			\label{fig_comput_time_aus_J}
		\end{subfigure}
		\caption{Computational and solution information for the Australian network (R-Cons: Redundant Constraints).}
		\label{fig_sol_aus_J}
	\end{figure}
	
	Similar to the 2-bus network, 100 FRs are also randomly generated to test the effectiveness of the RFR, where the percentage of random FRs containing the RFR and the initialised FR are 62\% and 0\%, respectively, where the number for RFR being less than 100\% is due to the local optimality in solving the BiLP-Max problem by the ITLP approach.
	
	Regarding the computational process, the proposed approach reports the RFR within three iterations, where each iteration consumes 130-180 seconds. Although the computational efficiency is acceptable if the RFR is used for planning purposes, we admit that it may not be suitable for intra-day or real-time operation purposes if the dispatch interval is less than 10 minutes. However, the practicality of the proposed approach can be significantly enhanced by the following measures.
	\begin{itemize}
		\item Further investigating the uncertainty modelling of line impedances to reduce the number of uncertain parameters and the conservativeness of uncertainty set.
		\item Applying parallel computing techniques (PCTs). Nothing that solving the BiLP-Max problem with various starting points of $(H,h)$ in each iteration can be paralleled, PCTs could considerably enhance the computational performance of the proposed approach.
	\end{itemize}
	
	\section{Conclusion}\label{sec_05}
	The paper studies how to calculate feasible region for DERs in an unbalanced distribution network while considering uncertainties from line impedances. The proposed approach, with an initialised FR, tries to seek the optimal solution by solving a non-convex bilinear optimisation problem and updating the RFR successively until reaching convergence, which has been demonstrated effective in two distribution networks. To make the proposed approach more practical, several research directions are open in this area and fall within our future research interest, including a) Improving computational efficiency in networks with more types of uncertainties and more DERs, particularly in properly modelling uncertainties and efficiently solving the BiLP-Max problem; 2) Investigating the implications of RFR on calculating operating envelopes.
	
	\bibliography{REFs_Power_Grid}
	
\end{document}